\documentclass[12pt]{article}

\usepackage{amssymb,amsthm,amsmath,amsfonts}
\usepackage{graphicx}
\usepackage{subcaption}

\usepackage{latexsym}
\usepackage[english]{babel}
\usepackage{tikz}

\usepackage{mathtools}

\newtheorem{theorem}{Theorem}
\newtheorem{lemma}[theorem]{Lemma}
\newtheorem{proposition}[theorem]{Proposition}
\newtheorem{corollary}[theorem]{Corollary}
\newtheorem{conjecture}[theorem]{Conjecture}

\theoremstyle{definition}
\newtheorem{definition}{Definition}
\newtheorem{example}{Example}
\newtheorem{observation}{Observation}
\newtheorem{notation}{Notation}

\usepackage{enumerate}
\usepackage[T1]{fontenc}
\usepackage[utf8]{inputenc}
\usepackage{authblk}

\title{Cumulative subtraction games}
\author[1]{Gal Cohensius \thanks{galcohensius@technion.ac.il}}
\author[1]{Urban Larsson \thanks{urban031@gmail.com}}
\author[1]{Reshef Meir \thanks{reshefm@ie.technion.ac.il}}
\author[2]{David Wahlstedt\thanks{david.wahlstedt@gmail.com}}
\affil[1]{Technion--Israel Institute of Technology, Haifa, Israel}
\affil[2]{G\"oteborg, Sweden}

\newcount\Comments  
\newcommand{\kibitz}[2]{\ifnum\Comments=1{\color{#1}{#2}}\fi}

\newcommand{\Mod}[1]{\ (\mathrm{mod}\ #1)}

\newcommand{\Zp}{\mathbb Z_{>0}}
\newcommand{\Znn}{\mathbb Z_{\ge 0} }
\newcommand{\Z}{\mathbb Z }

\newcommand{\opt}{\ensuremath{\mathrm{opt}}}

\newcommand{\tot}{\ensuremath{\mathrm{t}}}
\newcommand{\tr}{\ensuremath{\mathrm{tr}}}

\begin{document}
	
\maketitle
\begin{abstract}
We study a variation of Nim-type subtraction games, called Cumulative Subtraction (CS). Two players alternate in removing   pebbles out of a joint pile, and their actions add or remove points to a common \emph{score}. We prove that the zero-sum outcome in optimal play of a CS with a finite number of possible actions is eventually periodic, with period $2s$, where $s$ is the size of the largest available action. This settles a conjecture by Stewart in his Ph.D. thesis (2011). Specifically, we find a quadratic bound, in the size of $s$, on when the outcome function must have become periodic. In case of exactly two possible actions, we give an explicit description of optimal play.
\end{abstract}

\section{Introduction}
Two players, Alice and Bob, stand next to a single pile of $7$ pebbles, alternately taking pebbles from it.  They compete on who takes most pebbles.  However there is a restriction on the number of pebbles they may take each turn; on each turn they can take exactly $2$ or $3$ pebbles.  Now we ask the question:  if Alice starts, should she play \emph{greedily} and take $3$ or make a \emph{sacrifice} and take $2$?

In this paper we study generalizations of this game, called {\sc cumulative subtraction} (CS). CS is related to the famous game of {\sc nim}. It has similar type of moves, but a different winning condition.   We restrict attention to games with a single heap and a common finite action set of size at least 2.

\begin{definition}[{\sc cumulative subtraction}]\label{def:CS}
An instance of {\sc cumulative subtraction} (CS), $(S, x, p)$, is composed of a finite \emph{action set} $S$, where $|S| \ge 2$, a \emph{heap} of $x\in\Z_{\ge 0}$ pebbles, and a current \emph{score} $p$. In case $p=0$, the game is also denoted by $(S,x)$, and when the heap size is generic, the game is simply called $S$ (e.g. $S$ is viewed as a ruleset). It is a two player game, and the two players \emph{Positive} and \emph{Negative} take turns moving. A \emph{position} is denoted by $(x,p)$. 
A Positive's move is of the form $(x, p) \mapsto (x-s, p+s)$, for some $s\in S$, provided that $x-s\ge 0$.  A Negative's move is of the form $(x, p) \mapsto (x-s, p-s)$, for some $s\in S$, provided that $x-s\ge 0$. A position $(t,p_{\tot})$ is terminal if $t<\min S$. The \emph{result} of a game is the terminal score $p_{\tot}$. 
\end{definition}

We are interested in \emph{optimal play} of CS, which is a zero-sum game, where Positive is the `maximizer' and Negative is the `minimizer'.  Optimal play is reflected in the outcome function.

\begin{definition}[Outcome]\label{def:outcome}
The outcome of the game $(S,x)$ is 
$$ o(x)=
\begin{cases}
    \max_{s\in S}\{s-o(x-s)\}, &\mathrm{if} \; x \ge \min{S} \\
    0,         &\mathrm{otherwise}.
\end{cases}
$$
\end{definition}
Note that the maximizing action in Definition~\ref{def:outcome} might not be unique. However uniqueness is a convenient tool in proofs of optimal play, as is further explained via  Lemma~\ref{lem:big_before}.\footnote{When we use the term maximizing action we refer to an action that maximizes the outcome. When we say maximum action, we mean $\max S$.} To this purpose we define the opt-function. 

\begin{definition}[Optimal action]\label{def:opt}
Given a game $S$, the \emph{optimal action}, $\opt:\Z_{\ge \min S}\rightarrow S $, is a mapping from the set of non-terminal positions to the maximum action $s$, such that $o(x) = s-o(x-s)$. 
\end{definition}


Note that increasing the starting score by $p$ points will increase the outcome by $p$, but it will not change the optimal sequence of actions.

\begin{observation}
 The outcome is the (von Neumann \cite{vN} PSPE) game value if the initial score is 0, and Positive starts. 
\end{observation}

\begin{definition}[Game convergence]\label{def:conv}
A game $S$ converges at position $x>0$, if, for all positions $y\ge x$, $\opt(y)$ is constant, but $\opt(x-1)\ne \opt(y)$. This is denoted by $\xi(S)=x$. If there is no such $x$ then the game does not converge.
\end{definition}

\begin{observation}
If the game $S$ converges, then $\xi(S)\ge \max S$, because $\opt(\max S -1)< \opt(\max S)=\max S$. 
\end{observation}

 A sequence $(x_i)$ is \textit{periodic} if there is a $p$ such that, for all $i$, $x_i=x_{i+p}$. If $p$ is the smallest such number, then the sequence is periodic with period $p$. 
 
	\begin{definition}[Eventual periodicity]
	A function $g: \Znn\rightarrow\Z$ is eventually periodic, if there is a $p\in \Zp$, such that $g(x)=g(x+p)$, for all sufficiently large $x\in \Z$. It is eventually periodic with period $p '$, if $p'$ is the smallest $p$ such that $g(x)=g(x+p')$, for all sufficiently large $x\in \Z$.
	\end{definition}

Because of our convention that Positive starts, the relative number of actions that players will have throughout the game is:
either Positive and Negative play the same number of actions, or Positive has an extra turn.

\begin{definition}[Greedy action and sacrifice]\label{def:greedy}
Consider a game $(S,x)$. A greedy action is $\max\{s\in S\mid s\le x\}$, and a sacrifice is any action that is not greedy.
\end{definition}

Is the greedy action optimal in every position of CS game?  As we hinted in the first paragraph, the answer is no. 
\begin{example}\label{ex:2,3 game}
Consider the game $(S=\{2,3\}, x=7)$.  The optimal action is $\opt(7)=2$ because $o(5)= 1$ and $o(4)=3$ so $o(7)=\max\{-3+3, -1+2\}=1$.  That is, Positive's optimal strategy is a sacrifice, taking $2$ on her first action.
\end{example}

\begin{lemma}\label{lem:bounded_outcome}
    For all games $(S,x)$, the outcome is bounded between $0$ and the maximum action, i.e. $0 \le o(x) \le \max{S}$
\end{lemma}
\begin{proof}
    Since Positive plays at least the same number of actions as Negative plays, by playing greedily she guarantees a result of at least $0$.  Since Positive plays at most one more action than Negative, if Negative plays greedily he guarantees a result of at most $\max{S}$.
\end{proof}


\section{Contribution}
Our main result (see Section~\ref{sec:arbsup}) is that all CS games converge (to the maximum action), and thus the outcome of any CS is eventually periodic.  The results are

\begin{enumerate}

\item In Theorem~\ref{Thm:convergence} we give an upper bound on the convergence of any CS game.  The bound is quadratic in $\max{S}$.  

\item In Corollary~\ref{cor:periodicity} we prove that any game $S$, is eventually periodic with period $2\max S$.  That is, $o(x+2\max S)=o(x)$, for any large enough position $x$.  This is a proof of a conjecture by Stewart from \cite{S}.

\item In Theorem~\ref{thm:fullsup}, we fully solve the case where all actions up to $\max{S}$ are permitted (the case of full support). 
    
\item In Theorem~\ref{thm:2actions} we describe explicitly the optimal play for the class of games with exactly two actions, and in Corollary~\ref{cor:2action_outcomes} we specify the corresponding outcome.  In Corollary~\ref{cor:2action_explicit_convergence} we give an explicit formula for convergence.

\item Section~\ref{sec:truncsup} concerns so-called \emph{truncated games} (small actions are cut out). We were not able to solve the whole class, but we guide the readers towards a thrilling conjecture,  Conjecture~\ref{con:duality}.

\end{enumerate}

\section{CS with arbitrary support} \label{sec:arbsup} 
In this section we do not restrict $S$ beyond its definition, a finite set of size at least two.  Let us begin with a general lemma.
\begin{lemma} \label{lem:big_before}
    For any pair of sequences of optimal play actions, if one of the players, say Positive, switches the order between two actions such that the larger action is played before the smaller one, then this switch cannot decrease the outcome.
\end{lemma}
\begin{proof}
    By this switch, the opponent does not get any new playing possibilities, and thus the opponent's new optimal play is a sequence of actions that were available before the switch.
\end{proof}
By this lemma, without loss of generality, in this section we  assume that both players play non-increasing actions, and in particular, for each game, optimal play will give the unique sequence of actions as prescribed by the opt-function.\footnote{This idea is useful in this section, but other tools as for example the below ``complementary strategy'' does not use it.}
\begin{definition}
    Given a game $S$, the \textit{endgame} is the set of positions strictly smaller than $\max{S}$. A player \emph{enters} the endgame, playing from position $x \ge \max{S}$, if she plays action $a$ and $x-a < \max{S}$. The term \emph{endgame play} refers to the action that enters the endgame together with all subsequent moves.
\end{definition}

\begin{theorem} \label{Thm:convergence}
Consider a game $S$.  The upper bound on $\xi(S)$ is
\begin{align}
    \xi(S) \le 2(\max S)^2 \label{eq: convergence_bound}
\end{align}
\end{theorem}

\begin{proof}
      
    Consider play from some large  position until one of the players enters the endgame. Suppose that one of the players' strategy, say Positive's, consists in playing at least $\max S$ sacrifices before the endgame. 
    
    We will find a strategy $\sigma$ by Negative that produces a negative outcome. By Lemma~\ref{lem:bounded_outcome}, this  will imply that Positive's strategy cannot be optimal play. 
    
    Negative's strategy $\sigma$ is greedy play.
    
    Positive played at least $\max S$ sacrifices before the endgame.  There are two cases:
    \begin{itemize}
        \item[(i)] Positive enters the endgame
        \item[(ii)] Negative enters the endgame
    \end{itemize}
    
    In case (i), the score just before Positive enters the endgame is no more than $(-\max S)$.  In case (ii), the score after Negative entered the endgame is no more than $(-\max S)$.  This is true in both cases because Negative's greedy strategy $\sigma$ consisted exclusively of $\max S$ actions, whereas Positive has played at least $\max S$ sacrifices, so the outcome decreases by at least $1$ with each sacrifice. 
    
     
    By Lemma~\ref{lem:big_before} we may assume that players play non-increasing actions, i.e., at each stage of game, if more than one action produces the outcome, players will choose the largest of those actions. 
    Therefore, in case (i) Positive enters the endgame by a sacrifice, hence the score remains strictly smaller than 0 after Positive's action.  Thus, Negative assures an outcome strictly smaller than 0 (by Lemma~\ref{lem:bounded_outcome} applied to the players in reversed rolls). 
    
    In case (ii) when Negative enters the endgame,  Positive plays first below the heap size of $\max S$. By definition of the endgame, Positive can increase the score by at most $(\max S-1)$. 
    
    Thus, either way the outcome will be negative, and by the lower bound in Lemma~\ref{lem:bounded_outcome}, we have reached the desired contradiction.
    
    Therefore any optimal strategy, by either player,  must consist of less than $\max S$ sacrifices.

    This gives the bound in the theorem because Positive plays less than $\max S$ sacrifices in optimal play. Namely 
     \begin{align}
     \xi(S)&\le\max S+(\max S-1)(\max S-1)+(\max S-1)\max S\label{eq:bound}\\
     &\le 2(\max S)^2 \notag
     \end{align}
    where the terms in \eqref{eq:bound} represent: `upper bound on endgame size', `upper bound of the total size of Positive's sacrificing actions' and `upper bound of the total size of Negative's actions in response to Positive's sacrifices'. This concludes the proof.
 \end{proof}
 
    
    

\begin{lemma}\label{lem:converges->periodic}
Consider CS.  If the sequence of optimal actions converges, then the sequence of outcomes is eventually periodic.
\end{lemma}
\begin{proof}
If both players optimally play the same action $s$ from all sufficiently large heap sizes $x$, then $o(x)=s-o(x-s)=s-s+o(x-2s)=o(x-2s)$. 
\end{proof}

\begin{corollary}\label{cor:periodicity}
     Any game $S$, is eventually periodic with period $2\max S$.  That is to say, $o(x+2\max S)=o(x)$, for any large enough position $x$.
\end{corollary}
\begin{proof}
    Combine Theorem~\ref{Thm:convergence} with Lemma~\ref{lem:converges->periodic}.
\end{proof}

\section{CS with full support} \label{sec:full_support}
 Consider a CS where the set of possible actions contains all the integers from 1 up to $s_1$, i.e., $S = \{1,2,\ldots ,s_1\} $. We call this game CS with \emph{full support}.  In this game, optimal play is to play greedy at each position. 
	\begin{theorem}\label{thm:fullsup}
	In CS with \emph{full support}, the optimal play is $x$ for any position $x<s_1$ and $s_1$ for any position $x \ge s_1$.   That is, each CS with full support converges at $s_1$, and moreover its outcome is periodic with the pattern 
    \begin{align}
    (0,1,\ldots,s_1,s_1-1,\ldots, 1). \label{eq:fullsup}
    \end{align}
\end{theorem}
\begin{proof}
The proof is by induction.	For the base case, consider $0 \leq x \leq s_1$: when playing from position $x$, Positive takes all the pebbles, and thus $o(x) = x$.  When playing from positions $ x + s_1$, Positive's  optimal play is to take $s_1$ and negative takes the rest, thus $o(x+s_1) = s_1-o(x) = s_1 - x $.  It is Positive's optimal play since if she takes less than $s_1$ then Negative can take more than $x$.

Assume $k>0$ repetitions of the pattern (\ref{eq:fullsup}).  We study the next $s_1$ positions and show that the outcome in those positions will be exactly as in (\ref{eq:fullsup}).
\begin{align*}
o(x+2(k+1)s_1) &= s_1-o(x+2(k+1)s_1-s_1) \\ 
& = s_1 -o(x+2ks_1+s_1) \\
& = s_1 -o(x+s_1) \\
& = s_1 - (s_1 - x) \\
& = x
\end{align*}
For the following $s_1$ positions the outcome is 
\begin{align*}
o(x+ s_1 + 2(k+1)s_1) &= s_1-o(x+s_1+2(k+1)s_1-s_1) \\
& = s_1-o(x+2(k+1)s_1) \\
& = s_1-x
\end{align*}
\end{proof}

\section{CS with two actions}\label{sec:two_action}
In a game of just two possible actions,  $S = \{s_2, s_1\}$, with $s_1 > s_2$ we characterize the set of positions where it is optimal to sacrifice,  and this set will be called $X^*$ (see Definition~\ref{def:X^*} and Theorem~\ref{thm:2actions}). 

\begin{notation}
Let $\alpha = s_1-s_2$.
\end{notation}
We think of $\alpha$ as the size of the sacrifice a player makes by taking just $s_2$ instead of the greedy action $s_1$.

\begin{definition}\label{def:X^*}
Let $\Delta = \{0,1,\ldots , \alpha-1 \}$. For each $i \in \Z_{>0}$, such that
	\begin{align}\label{eq:sacrifice}
		is_2 > (i-1)s_1,
	\end{align}
	let
	\begin{align}\label{eq:x star}
	X^*(i) = \{is_2+(i-1)s_1 + \delta \mid \delta \in \Delta \}, 
	\end{align}
and otherwise $X^*(i)=\varnothing$.  Let 
 $$X^*= \bigcup_{i \in \Z_{>0}} X^*(i)$$
\end{definition}

Inequality (\ref{eq:sacrifice}) means that $i$ sacrifices is worth more than $(i-1)$ greedy actions.  

A simple observation is that no player can benefit by playing more than $\left\lfloor\frac{s_1}{\alpha}\right\rfloor$ sacrifices. 
\begin{lemma}\label{lem:sacrbound}
    No player benefits by playing more than
    $i_{\max}= \left\lfloor\frac{s_1}{\alpha}\right\rfloor = 1+\left\lfloor\frac{s_2}{\alpha}\right\rfloor$ 
    sacrifices. 
\end{lemma}
\begin{proof}
Suppose that $i\in \Zp$ counts the number of sacrifices by Positive, and assume $is_2 < (i-1)s_1$, where Negative plays $i-1$ greedy actions. Then  $o(x)<0$, which is impossible by Lemma~\ref{lem:bounded_outcome}. Hence, for any optimal strategy we must have $is_2 \ge (i-1)s_1$. Therefore, $s_1\ge i\alpha$, which implies the lemma since $i$ is an integer. Moreover, observe that $\frac{s_1}{\alpha}=\frac{s_1 + s_2 - s_2}{\alpha} = 1+\frac{s_2}{\alpha}$.
\end{proof}

Note that $i_{\max}$ is the largest $i$ such that (\ref{eq:sacrifice}) holds.  E.g., in Example~ \ref{ex:game 5,7}, $i_{\max} = 3$. We will see that $i_{\max}-1$ is the maximum number of sacrifices a player can beneficially make, to win an extra turn.  In Example~\ref{ex:game 5,7}, 2 sacrifices are still beneficial since $3\cdot5 > 2\cdot7$ however 3 sacrifices are not since $4\cdot5 \ngtr 3\cdot7$.

\begin{example}\label{ex:game 5,7}
    Consider the game $S=\{5,7\}$.  The only positions where playing $s_2=5$ is strictly better than playing $s_1=7$ are $X^*=\{5,6,17,18,29,30\}$.  Table~\ref{table:(5,7)} presents the optimal actions and outcomes for the first 55 positions of $S=\{5,7\}$.
\end{example}

\begin{center}
\begin{table}[ht!]
\begin{tabular}{|c|| c c c c c c c|c c c c c c c|} 
\hline
$x$			&0 &1 &2 &3 &4 &5 &6   &7 &8 &9 &10&11&12&13 \\ 
\hline
$\opt(x) $		&- &- &- &- &- &\color{red}5 &\color{red}5   &7 &7 &7 &7 &7 &7 &7 \\ 
\hline
$o(x)$		&0 &0 &0 &0 &0 &5 &5   &7 &7 &7 &7 &7 &2 &2 \\
\hline\hline
$x$			&14&15&16&17&18&19&20  &21&22&23&24&25&26&27 \\
\hline
$\opt $		&7 &7 &7 &\color{red}5 &\color{red}5 &7 &7   &7 &7 &7 &7 &7 &7 &7 \\ 
\hline
$o(x)$		&0 &0 &0 &3 &3 &5 &5   &7 &7 &7 &4 &4 &2 &2 \\
\hline\hline
$x$			&28&29&30&31&32&33&34  &35&36&37&38&39&40&41 \\
\hline
$\opt(x) $		&7 &\color{red}5 &\color{red}5 &7 &7 &7 &7   &7 &7 &7 &7 &7 &7 &7 \\ 
\hline
$o(x)$		&0 &1 &1 &3 &3 &5 &5   &7 &6 &6 &4 &4 &2 &2 \\
\hline\hline
$x$			&42&43&44&45&46&47&48  &49&50&51&52&53&54&55 \\ 
\hline
$\opt(x) $		&7 &7 &7 &7 &7 &7 &7   &7 &7 &7 &7 &7 &7 &7 \\ 
\hline
$o(x)$		&0 &1 &1 &3 &3 &5 &5   &7 &6 &6 &4 &4 &2 &2 \\
\hline
\end{tabular}
\caption{Positive's (largest) optimal action, and the outcome for CS with action set $S=\{5,7\}$, starting from position $x$.}
\label{table:(5,7)}
\end{table}
\end{center}

Next, we develop a tool, Positive's `complementary strategy', which gives a lower bound on the result, and it equals the outcome if Negative plays optimally (see Lemma~\ref{lem:complementary}). 

\begin{definition} [Complementary strategy]
 Positive plays on the first action $s_2$.  From now on Positive's actions complement Negative's actions modulo $(s_1+s_2)$, that is, if Negative's action is $s_1$ then Positive's action is $s_2$ and vice versa.
 \end{definition}
 
\begin{lemma}\label{lem:complementary}
From position $x\in X^*(i)$ Positive's optimal play is the complementary strategy and Negative's optimal play is the greedy strategy.  The outcome is 
 \begin{align}\label{eq:>}
     o(x) = is_2 - (i-1)s_1 > 0.       
 \end{align}
\end{lemma}
\begin{proof}
If Positive plays the complementary strategy, this produces at least the result in (\ref{eq:>}). Moreover, Positive gets $i$ turns by the complementary-strategy vs. Negative's $(i-1)$ turns. 
By definition of $ X^*(i)$, we let $x=is_2+(i-1)s_1+\delta$, for some  $ \delta \in \{0,...,\alpha-1\} $.  Suppose that Positive deviates from the complementary strategy and plays at least one greedy action, then Negative can play only greedy actions and play the last turn.  By estimating the number of remaining pebbles for Positive,   $x-s_1-(i-1)s_1 = is_2-s_1 + \delta < (i-1)s_2$, Positive can play at most $i-1$ actions. If this were optimal play, the outcome would be at most 0, which contradicts (\ref{eq:>}).  Suppose that Negative deviates from the greedy strategy, then Positive still plays the Complementary-strategy and gets an extra $\alpha$ for each deviation of Negative.
\end{proof}

Equivalently, for $x\in X^*(i)$,  
\begin{align}\label{eq:outx*}
o(x) = s_1-i\alpha, 
\end{align}
and a consequence of this is Lemma~\ref{lem:duality}. Look at Table \ref{table:(5,7)}. For any position in $X^*(i)$, the outcome equals the number of consecutive positions with outcome $0$ immediately to the left of $X^*(i)$. For example  $o(17)=3$, and the relevant positions are 14, 15, 16. This holds for any $S=\{s_2,s_1\}$. The outcome is $0$ in those positions since both players will play $s_1$ until the game ends, and they will have equal numbers of turns.

\begin{lemma} \label{lem:duality}
    Suppose that $x= \min\{ X^*(i)\}$, for any $1\le i\le i_{\max}$, and $y$ is such that $x-o(x)\le y<x$. Then 
    \begin{align}\label{eq:duality}
    o(y) = 0,
    \end{align}
    and the optimal action is $\opt(y)=s_1$. \\
\end{lemma}
\begin{proof}
By Lemma~\ref{lem:complementary}, since $o(x)=is_2 - (i-1)s_1$, we get $$2(i-1)s_1\le y<is_2+(i-1)s_1,$$ where the upper bound is by $x= \min\{ X^*(i)\}$. If Positive starts by playing $s_2$, then by the lower bound, Negative can play $i-1$ turns of $s_1$, whereas by the upper bound, Positive can play at most $(i-1)$ actions in total. Hence the result is negative, which is not optimal, by Lemma~\ref{lem:bounded_outcome}. 

On the other hand, by the lower bound if both players play greedily, the result is 0. This is the outcome, because Positive cannot do better, and by Lemma~\ref{lem:bounded_outcome}, neither can Negative. 
\end{proof} 
\begin{proposition}
If $(2i-1)s_1\le y\le i(s_1+s_2)$, for any $1\le i\le i_{\max}$, then $o(y)=s_1$, and $\opt(y)=s_1$.
\end{proposition}

\begin{proof}
Combine Lemma~\ref{lem:duality} with Lemma~\ref{lem:bounded_outcome}, to see that $\opt(y)=s_1$, which implies $o(y)=s_1$, by Lemma~\ref{lem:duality}. 
\end{proof}
The following result is used in the second part of the proof of Theorem~\ref{thm:2actions}. 
\begin{lemma}\label{lem:bounded diff}
Consider $S = \{s_2,s_1\}$ and $\alpha = s_1 - s_2$. If $x\not\in X^*,$ and $x \ge\alpha$ then 
\begin{align}\label{eq:alpha} 
    o(x) - o(x-\alpha) \le \alpha.
\end{align}
\end{lemma}
    \begin{proof}
	We study the function 
	$$\eta(x) := \alpha + o(x - \alpha) - o(x), $$
    and  show that $\eta(x)\ge 0$, if $x\not\in X^*$ and $x\ge \alpha$.
    We think about $o(x)$ as the outcome when Positive starts, and $-o(x)$ as the outcome when Negative starts. It suffices to show that, for all plays by Negative from $x$, there is a response by Positive such that the inequality (\ref{eq:alpha}) holds.\\

    \noindent Case 1: If there is no move from $x$ (because $s_2>x$) then $\eta(x)=\alpha\ge 0$. \\
    
    \noindent Case 2: If there is a move from $x$, but no move from $ x-\alpha$, then $x < s_1$; thus $x\in X^*(1)$, and the only possible action is $s_2$, then $\eta(x)=\alpha-o(x)=\alpha-s_2 > 0$ since $2s_2 > s_1$. \\
    
    \noindent Case 3: If there is a move from $x$, and a move from $ x-\alpha$, then
  	\begin{enumerate}
		\item If Negative plays optimally $s_1$ from $x$, and Positive plays $s_2$ from $x-\alpha$, we get
		\begin{align*}
        \eta(x) &\ge \alpha + o(x-s_1) -s_1- o(x-\alpha-s_2) +s_2 \\ &= o(x-s_1) - o(x-s_1)\\ &= 0
        \end{align*}
        
		\item If Negative plays optimally $s_2$ from $x$, and Positive plays $s_2$ from $x-\alpha$, we get
		$$\eta(x)\ge \alpha + o(x-s_2) - o(x-\alpha-s_2), $$
        and we note that, if Negative has no move from $x-\alpha-s_2 = x-s_1$, then this implies $\eta(x)\ge 0$.  Assume Negative has a move then there are two cases:
		\begin{enumerate}[2.1]
			\item On the second move, if Negative plays optimally $s_2$, and Positive plays $s_1$, we get
			$$ \eta(x)\ge \alpha-s_2+o(x-s_1-s_2) +s_1-o(x-s_2-s_1) = 2\alpha > 0$$			
			\item On the second move, if Negative's optimal move is $s_1$, and Positive responds with $s_1$, we get
			$$ \eta(x)\ge \alpha + o(x-s_1-s_1)-o(x-s_2-s_1) = \eta(x-s_2-s_1)$$
            and since, by definition of $X^*$, if $x \not\in X^*$ then $x-s_1-s_2 \not\in X^*$. Therefore $\eta(x-s_2-s_1) \ge 0$ by induction.
		\end{enumerate}
	\end{enumerate}

    This concludes the proof of inequality (\ref{eq:alpha}).
    \end{proof}

The following theorem is the main result for CS with two actions.
\begin{theorem}\label{thm:2actions}
Let the action set be $S = \{s_2, s_1\}$, with $s_1 > s_2$.  Then $\opt(x) = s_2$ if and only if $x\in X^*$.  
 \end{theorem}
 \begin{proof}
 We prove that $\opt(x)=s_2$ if and only if $x\in X^*$.  
 The proof is split into two cases. 
\begin{enumerate}[(i)]
    \item $2s_2 \le s_1$: greedy play is optimal.
    \item $2s_2 > s_1$:  sacrifice is optimal if and only if $x\in X^*$ 
\end{enumerate}
In Case (i), $i_{\max}=1$, which means that it is never beneficial to sacrifice.  Thus, in this case the optimal play, game convergence and  periodicity is analogous to the full support case, (Theorem~\ref{thm:fullsup}). 

Consider Case (ii). For the direction ``$x\in X^*$ implies $\opt(x)=s_2$'', by Lemma~\ref{lem:complementary} we know that optimal play from position $x\in X^*$ is given by Positive's complementary strategy, which starts by playing action $s_2$.
  
 
 
The proof of the reverse direction ``$\opt(x)=s_2$ implies $x\in X^*$'' uses Lemma~\ref{lem:bounded diff}.  
We prove by induction that for each position $x\not\in X^*$, if $x\ge s_1$ then $s_1$ is an optimal move.  We begin by stating the base case. 

Consider $x\in \{0,\ldots ,2s_1-1\}$. If $x<s_2$, no action is available. (For positions $x \in X^*(1)$, only action $s_2$ is available, so it is optimal.) For positions $x\in \{s_1,\ldots, 2s_1-1\}\subset \Z \setminus X^*$, $s_1$ is the unique optimal action, since it can be countered with at most one $s_2$ action before the end of play, and $2s_2+s_1\ge 2s_1$, by Case (ii). 
    
Assume next that $x\ge 2s_1$. It suffices to prove that playing $s_1$ is weakly better than playing $s_2$, i.e.   
\begin{enumerate}
\item[ ] if $x\not\in X^*$, then $-o(x-s_1) + s_1 \ge -o(x-s_2) + s_2$,
\end{enumerate}
or equivalently
\begin{enumerate}
\item[] if $x\not\in X^*$, then $\alpha \ge o(x-s_1)-o(x-s_2)$.
\end{enumerate}
    
There are three cases, depending on whether $x-s_1$ or $x-s_2$ belongs to $X^*$ respectively. Note that both cannot belong to $X^*$, because $x-s_2-(x-s_1)=\alpha$, and, for all $i$, $X^*(i)$ contains at most $\alpha -1$ consecutive numbers (and more than $s_1$ numbers separate two disjoint sets $X^*(i)$ and $X^*(j))$. 
    
\begin{enumerate}
\item $x-s_1\in X^*$, $x-s_2\not\in X^*$ 
\item $x-s_1\not\in X^*$, $x-s_2\not\in X^*$ 
\item $x-s_1\not\in X^*$, $x-s_2\in X^*$ 
\end{enumerate}
    
For 1., use the statement of the theorem as induction hypothesis, that is $s_2\in \opt(x-s_1)$ and $s_1\in \opt(x-s_2)$. We get
\begin{align*}
o(x-s_1) - o(x-s_2) &= -o(x-s_1-s_2)+s_2+o(x-s_2-s_1)-s_1\\
= -\alpha\\
<\alpha. 
\end{align*}
 For 2., use induction to conclude $s_1\in \opt(x-s_1)$ and $s_1\in \opt(x-s_2)$. We get
\begin{align*}
o(x-s_1) - o(x-s_2) &= -o(x-s_1-s_1)+s_1+o(x-s_2-s_1)-s_1\\
&= -o(x-s_1-s_2-\alpha)+o(x-s_2-s_1)\\
&\le\alpha, 
\end{align*}
if Lemma~\ref{lem:bounded diff} applies, i.e. if $x-s_1-s_2\not\in X^*$. Thus, in this case we are done.  

The other case is whenever $x-s_1-s_2\in X^*$. Since $x\not\in X^*$, this case happens if and only if $x-s_2-s_1 \in X^*(i_{\max})$. By (\ref{eq:outx*}) and Lemma~\ref{lem:bounded_outcome}, in this case, 
\begin{align}
    o(x-s_2-s_1)-o(x-s_2-s_1-\alpha) &\le s_1-\alpha i_{\max}-0\\
    &<\alpha,\label{eq:something}
\end{align}
 where the inequality \eqref{eq:something} is by $i_{\max}=\lfloor\frac{s_2}{\alpha}\rfloor+1>\frac{s_2}{\alpha}$.\\

 For 3., consider first the case $i < i_{\max}$. We use the `duality' (\ref{eq:duality}) between outcomes and number of consecutive positions with outcome $0$ just below $X^*(i)$. Indeed, in this case, Lemma~\ref{lem:duality} implies that there are at least $\alpha$ such consecutive positions with outcome $0$, that is, $o(x-s_2)-o(x-s_1)=s_1-i\alpha - 0 > \alpha $, and so $$o(x-s_1)-o(x-s_2)<-\alpha\le \alpha .$$ 
 
 The remaining case is for $x-s_2\in X^*(i_{\max})$. We use that $o(x-s_2)\ge 0$, and prove that $o(x-s_1) = \alpha$. This suffices, to prove the theorem.
 
 Let us first sketch the idea, of this final part of the proof. In fact, by our previous items, playing optimally from $x-s_1$, there will be an even number of greedy actions, namely $2i_{\max}$, of which the last one is $s_2$. This follows because, none of the greedy actions will end up in $X^*$, and we showed already that $s_1$ is optimal if a player does not start in $X^*(i)$, with $i<i_{\max}$. Indeed, this gives the outcome $\alpha$. 
 
 To finish the proof, let us justify the claim in the previous paragraph. Since $x-s_2\in X^*(i_{\max})$, we have that $x-s_1= i_{\max}(s_1+s_2)-s_1-\alpha +\delta_1$, for some $\delta_1\in\{0,\ldots , \alpha-1\}$. Let $X^*=\{i_{\max}(s_1+s_2)-s_1-\alpha +\delta_2\}$, where $\delta_2\in\{0,\ldots , \alpha-1\}$. If we show that, for all $0<j$, $\delta_2\not\equiv -\alpha +\delta_1-js_1 \pmod{s_1+s_2}$, the claim follows.  So, assume that there is an integer $k$, such that $-\alpha-js_1+k(s_1+s_2)\in\{-\alpha+1,\ldots , \alpha-1\}$, for some $j$. That is, we have that $(k-j)s_1+ks_2=-js_1+k(s_1+s_2)\in\{1,\ldots , 2\alpha-1\}$. By $s_1-s_2=\alpha$, this implies that $k-j=k+1$, i.e. $j=-1$. But, this contradicts the assumption that $j>0$. Therefore, if the players play greedily, they will never play to the set $X^*$. In the proofs of item 1., item 2., and the first part of item 3., we already proved that, for smaller positions than $X^*(i_{\max})$, sacrificing is optimal if and only if playing from $X^*$.
 
 This proves the theorem.
\end{proof}


Let us denote by $[x]_{y}$ the smallest non-negative number congruent to $x$ modulo $y$.
\begin{corollary}\label{cor:2action_outcomes}
The outcomes of the game $S=\{s_2,s_1\}$ are
$$ o(x)=
\begin{cases}
    is_2-(i-1)s_1 = s_1-i\alpha,    &\text{if} \; x\in X^*(i) \\
    0,                              &\text{if} \; y-s_1+i\alpha\le x<y, \text{where}~  y\in X^*(i) \\
    o(y),                           &\text{if there is} \; y\equiv x\Mod{2s_1}, \\                                                        &\text{s.t.} ~y\in X^*(i) \text{, with} ~ y<x \\
    s_1 - o(x-s_1),                 &\text{for all} \; [x]_{2s_1} \in \{s_1,\ldots ,2s_1-1\}
\end{cases}
$$

\end{corollary} 

\begin{proof}
    This follows from proof of Theorem~\ref{thm:2actions}.
\end{proof}
        In particular, the periodic outcome pattern, at convergence, is obtained by applying $i = i_{\max}$. See also Figure~\ref{fig:out}.  
         
         Note that the first three items concern the outcomes of the positions in the congruence classes $0,\ldots , s_1-1 \pmod{2s_1}$ and the last item concerns the `anti-symmetric' part among the heap sizes $s_1,\ldots , 2s_1-1 \pmod{2s_1}$. The third item shows that once the outcomes for positions in $X^*(i)$ have been computed, then they stabilize, for congruent larger heap sizes modulo $2s_1$.
         
         Another consequence is that if $[x]_{2s_1} \in \{s_1,\ldots, 2s_1-1\}$, then $s_1\in \opt(x)$.

\begin{corollary}\label{cor:2action_explicit_convergence}
Consider CS with two possible actions, $S=\{s_2,s_1\}$, with $s_1>s_2>s_1/2$. Then the largest heap size for which Positive can play the smaller action $s_2$ until the game ends, and obtain the optimal play outcome, is 
$$s_1 \left\lfloor{\frac{s_2}{\alpha}}\right\rfloor + s_2 \left\lfloor{\frac{s_1}{\alpha}}\right\rfloor+\alpha-1$$
\end{corollary}
\begin{proof}
This follows by Theorem~\ref{thm:2actions}.
\end{proof}

Note that the formula in Corollary~\ref{cor:2action_explicit_convergence} implies explicit game convergence at $\xi(S)=s_1 \left\lceil{\frac{s_2}{\alpha}}\right\rceil + s_2 \left\lceil{\frac{s_1}{\alpha}}\right\rceil-2s_2=(s_1+s_2) \left\lceil{\frac{s_2}{\alpha}}\right\rceil-s_2$; for example in case $s_1=s_2+1$, then $\xi(x)=2s_2^2$.

In Figures~\ref{fig:optact} and \ref{fig:out} we sketch the optimal actions modulo $s_2+s_1$ and the outcomes modulo $2s_1$, of the two-action games with $2s_2\ge s_1$.\vspace{2cm}

\begin{figure}[ht!]
\centering
\begin{tikzpicture}[thick,scale=.72]%
   \foreach \x in {90, -130,-50} { 
   \draw[-]  (\x:2.9) -- (\x:3.1) node{}; }   
   \draw (0, 0) node{\small{Pile size}};
   \draw(0,2.27) node{$0$};
   \draw(1.5,-1.55) node{$s_2$};
   \draw(-1.5,-1.55) node{$s_1$};
   \draw(-4.7,0) node{$s_1\in \opt(x)$};
   \draw(4.7,0) node{$s_1\in \opt(x)$};
   \draw(0.3,-3.8) node{$\opt(s_2 + \delta)=\{s_2\}$};
     
\draw[thin] (0,0) circle (3 cm) node{ };
\draw (0,4) node{\small{\bf Optimal actions before convergence modulo $(s_1+s_2)$}};
\end{tikzpicture}\caption{Optimal actions before convergence, for pile sizes modulo $(s_1+s_2)$. The positions in $X^*$ are of the form $x+(s_1+s_2)i$, for $0\le i< \lfloor\frac{s_1}{\alpha}\rfloor$, and where $s_2 \le x < s_1$. The pile sizes are pictured on the inside and the optimal actions on the outside.}\label{fig:optact}
\end{figure}
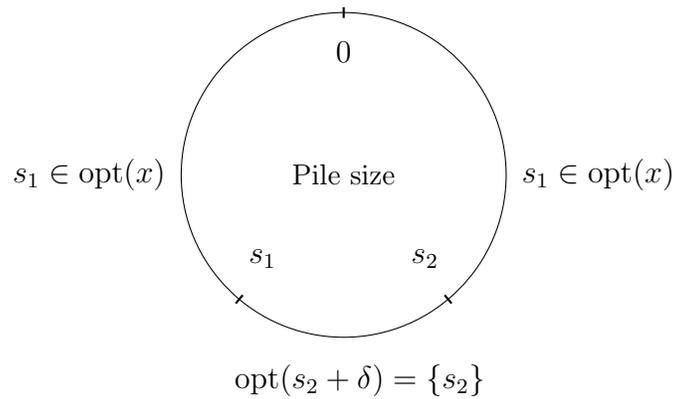

\clearpage
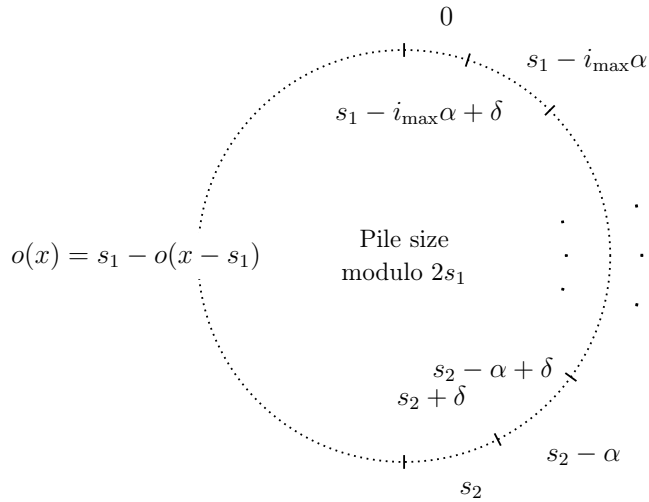
\begin{figure}[ht!]
\centering
\scalebox{0.87}{
\begin{tikzpicture}[thick,scale=1.05]%
\begin{scope}[xshift=0cm, yshift=10 cm]
\draw (0, 0) node{\small{Pile size}};
    \foreach \x in {90, 150, 270,330}{ 
    \draw[-]  (\x:2.9) -- (\x:3.1) node{};}

   \draw (0,4) node{{\bf The outcomes of the first $2s_1$ positions }};
   \draw(0,2.27) node{$x=0$};
     \draw(1.7,-1.15) node{$x=s_2$};
      \draw(3.35,-2.95) node{$o(s_2+\delta)=s_2$};
      \draw(4.45,1.55) node{$o(x)=0$};
      \draw(-4.45,-1.55) node{$o(x)=s_1$};
       \draw(-4.2,2.7) node{$o(2s_1-\delta) = s_1-s_2$};
   \draw[dotted] (0,0) circle (3 cm) node{ };
   
\end{scope}
\begin{scope}[xshift=0cm, yshift=0 cm]
   \draw[dotted] (0,0) circle (3 cm) node{ };
   \draw (0, 0.25) node{\small{Pile size}}; 
   \draw (0, -0.25) node{\small{modulo $2s_1$}}; 
   \foreach \x in {90,-90+6*27,-90+5*27,-36,-63,-90}{ 
    \draw[-]  (\x:2.9) -- (\x:3.1) node{};}
    
     \foreach \x in {12.3,12,0,0.3,-12,-11.7}{ 
    \draw[-]  (\x:2.34) -- (\x:2.38) node{};}
  \foreach \x in {12.3,12,0,0.3,-12,-11.7}{ 
    \draw[-]  (\x:3.44) -- (\x:3.48) node{};}
       \foreach \x/\xtext in {
            84/s_1-i_{\max}\alpha+\delta,
            -52/s_2-\alpha+\delta,
            -79/s_2+\delta
        }
      \draw (\x:2.1cm) node {$\xtext$};
       \foreach \x/\xtext in {
       180/o(x)=s_1-o(x-s_1),
       47/s_1-i_{\max}\alpha,
       -48/s_2-\alpha
        }
      \draw (\x:3.9cm) node[fill=white] {$\xtext$};
 \foreach \x/\xtext in {
       80/0,
      -74/s_2
        }
      \draw (\x:3.55cm) node[fill=white] {$\xtext$};
	   \draw (0,4.45) node{{\bf The outcomes at convergence modulo $2s_1$}};
\end{scope}
\end{tikzpicture}
}\caption{Initial outcomes (top) and outcomes at convergence (bottom) for pile sizes modulo $2s_1$, for 2-action games. The pile sizes are pictured on the inside and the outcomes on the outside of the respective circle.}\label{fig:out}  
\end{figure}

 

\clearpage

\clearpage
\section{CS with truncated support}\label{sec:truncsup}
In Section~\ref{sec:full_support} we have a simple proof for the full support case, and this might lead one to think that the generalized case of \emph{truncated support} is similarly simple.  However we do not yet  understand the full class of truncated support games.  So far, our efforts lead us to the intriguing Conjecture~\ref{con:duality}.

\begin{definition}[Truncated support games] 
Consider a game $S$, with $m=\max S\ge 2$, of the form $S=\{a, a+1,\ldots , m\}$, where $a\in \{1, m-1\}$, so that $|S|=m-a+1$ and we say that $S$ is ($a-1$)-truncated.
\end{definition}

The truncated support games includes as special cases both all games with full support ($a=1$, 0-truncated) and some games with two actions ($a=m-1$, $m-2$-truncated) which are the games that have the slowest convergence.  

For each $a$, we estimate in which interval of size $2m$, optimal play converges to the maximal action $m$. 
\begin{definition}[Convergence interval]
    If $\xi(S)\in\{2(j-1)m,\ldots , 2jm\}$, then the interval of convergence is $\tr^m_a=j$. Let $\tr^m$ denote the sequence of the form $\tr^m=(\tr^m_a)_{a=1}^{m-1}$. 
\end{definition}

\begin{example}\label{ex:truncated suport m=5}
When $m=5$, then the sequence is $\tr^5 = (\tr^5_1, \tr^5_2, \tr^5_3, \tr^5_4) = (1, 2, 2, 4)$. 
 Here, the first entry $\tr^5_1 =1$ shows that when $S=\{1, 2, 3, 4, 5\}$ (which is the full support game of size 5), then the convergence to greedy action in optimal play occurs already in the \emph{first} interval of size $10$ (convergence at position $x=10$). The last entry, $\tr^5_4 =4$, concerns the game $S = \{4, 5\}$, and convergence occurs by the $4^{\rm th}$ interval of size $10$. 
\end{example}

\begin{table}[ht]
\centering
\begin{tabular}{|c||c|c|c|c|c|c|c|c|c||c|} 
\hline
$m \setminus a$&1&2&3&4&5&6&7&8&9& $\#x$\\ \hline
$2 $&1&&&&&&&&&1\\\hline
$3$ &1&2&&&&&&&&2\\\hline
$4$&1&2&3&&&&&&&3\\\hline
$5$&1&2&2&4&&&&&&3\\\hline
$6$&1&2&2&3&5&&&&&4\\\hline
$7$&1&2&2&2&3&6&&&&4\\\hline
$8$&1&2&2&2&3&4&7&&&5\\\hline
$9$&1&2&2&2&2&3&4&8&&5\\\hline
$10$&1&2&2&2&2&3&3&5&9&5\\
\hline
\end{tabular} 
\caption{The convergence interval $\tr^m_a$ for every $m \in \{2, \ldots, 10\}$ and every $a < m$}
\label{tabel:truncated}
\end{table}

The $a^{\rm th}$ column in Table~\ref{tabel:truncated} shows the convergence for $(a-1)$-truncated support games, for $m\in\{2,\ldots, 10\}$. The $\# x$ column is the number of unique values of $\tr^m_a$.

\noindent From this table alone, for $a\ge 2$, the sequence of number of occurrences is non-increasing. But this is not true in general. To obtain some more insight, we plot the entries for $m=25,50,100$.  Via early observations, these pictures seem to converge to some function of the form $\frac{A}{\sqrt{B-x}}$; the appearing symmetry has a precise formulation, explained in the below conjecture.

\begin{figure}[ht] 
\includegraphics[width = .30\textwidth]{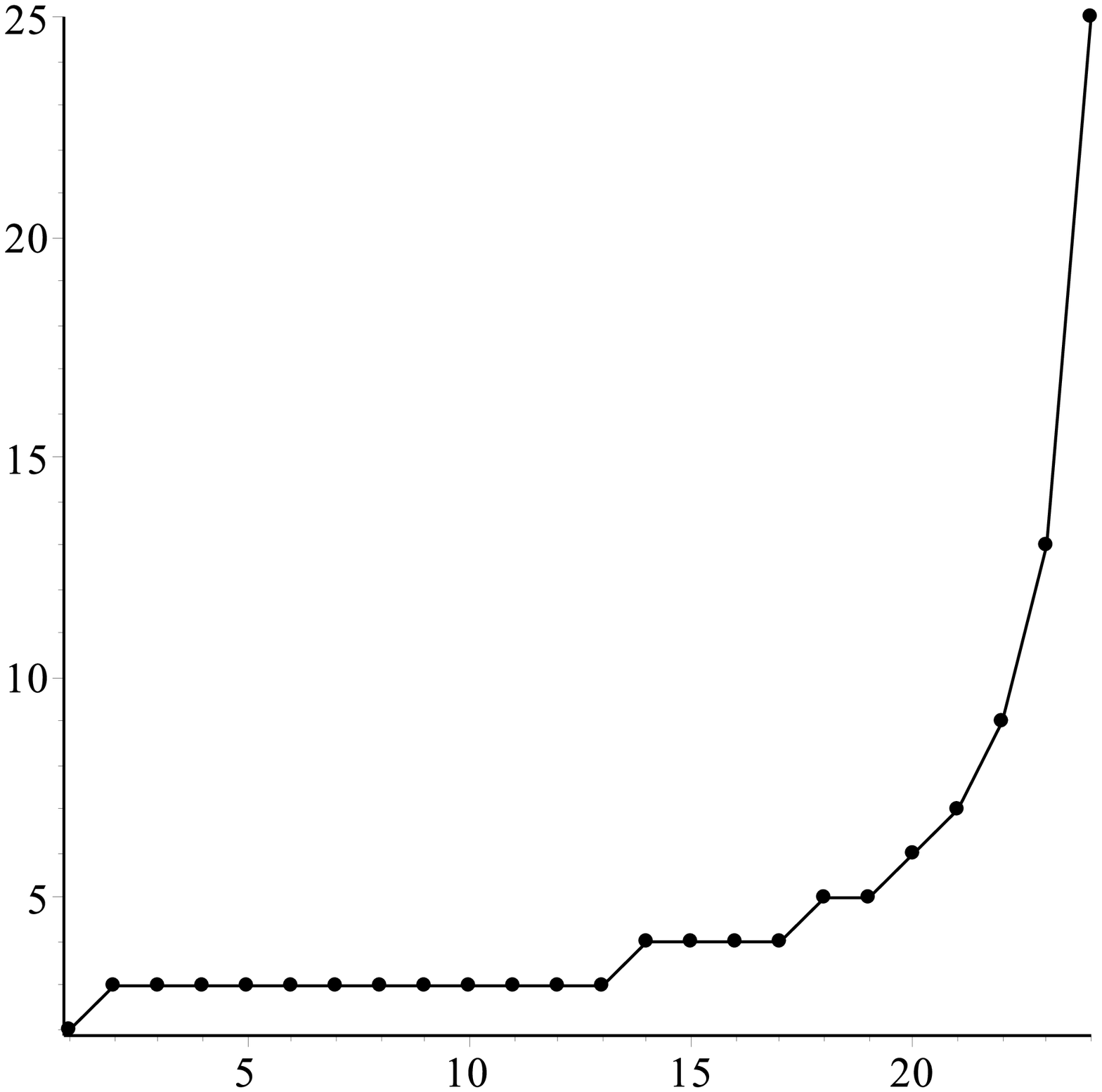}
\includegraphics[width = .30\textwidth]{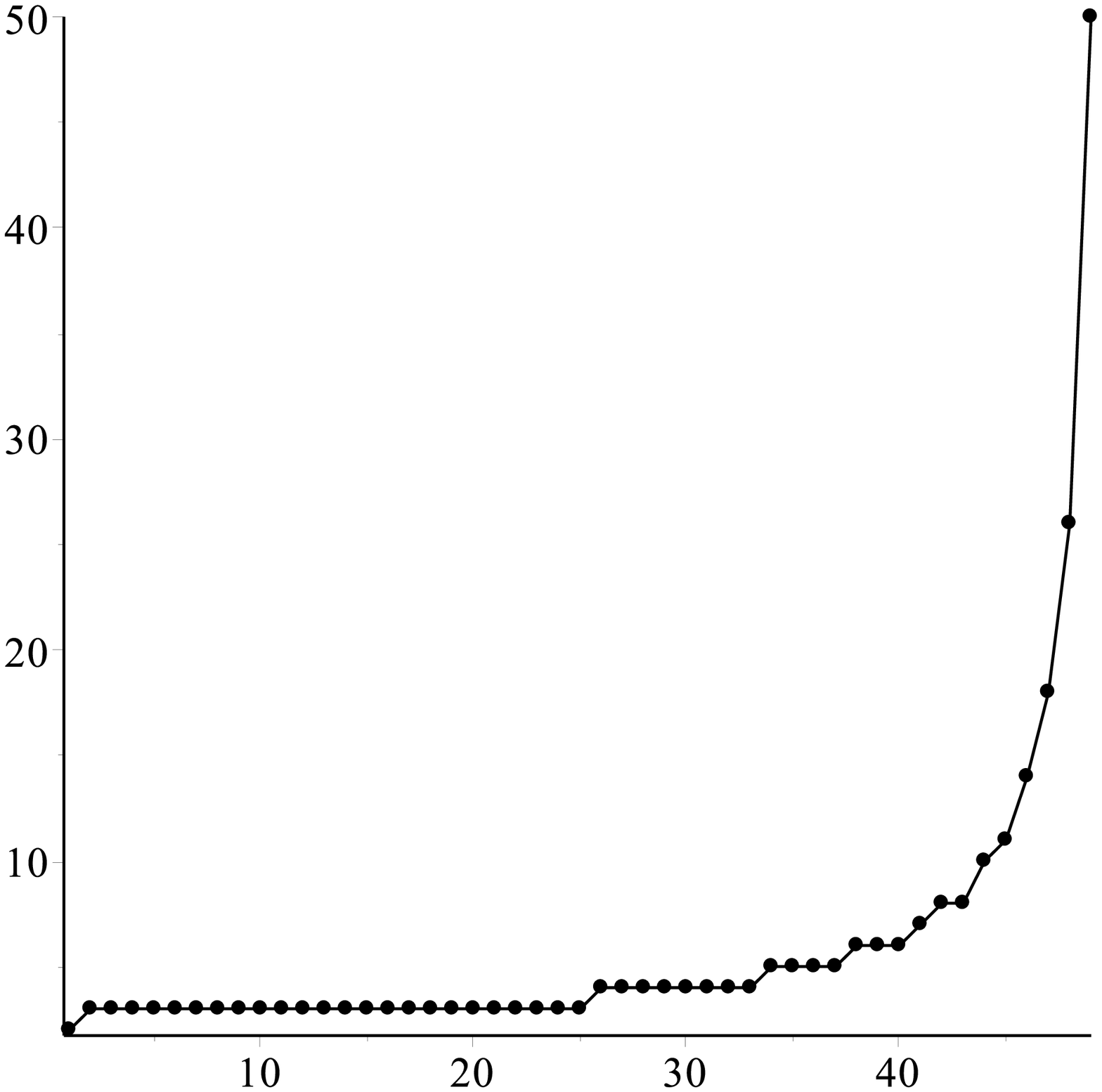}
\includegraphics[width = .30\textwidth]{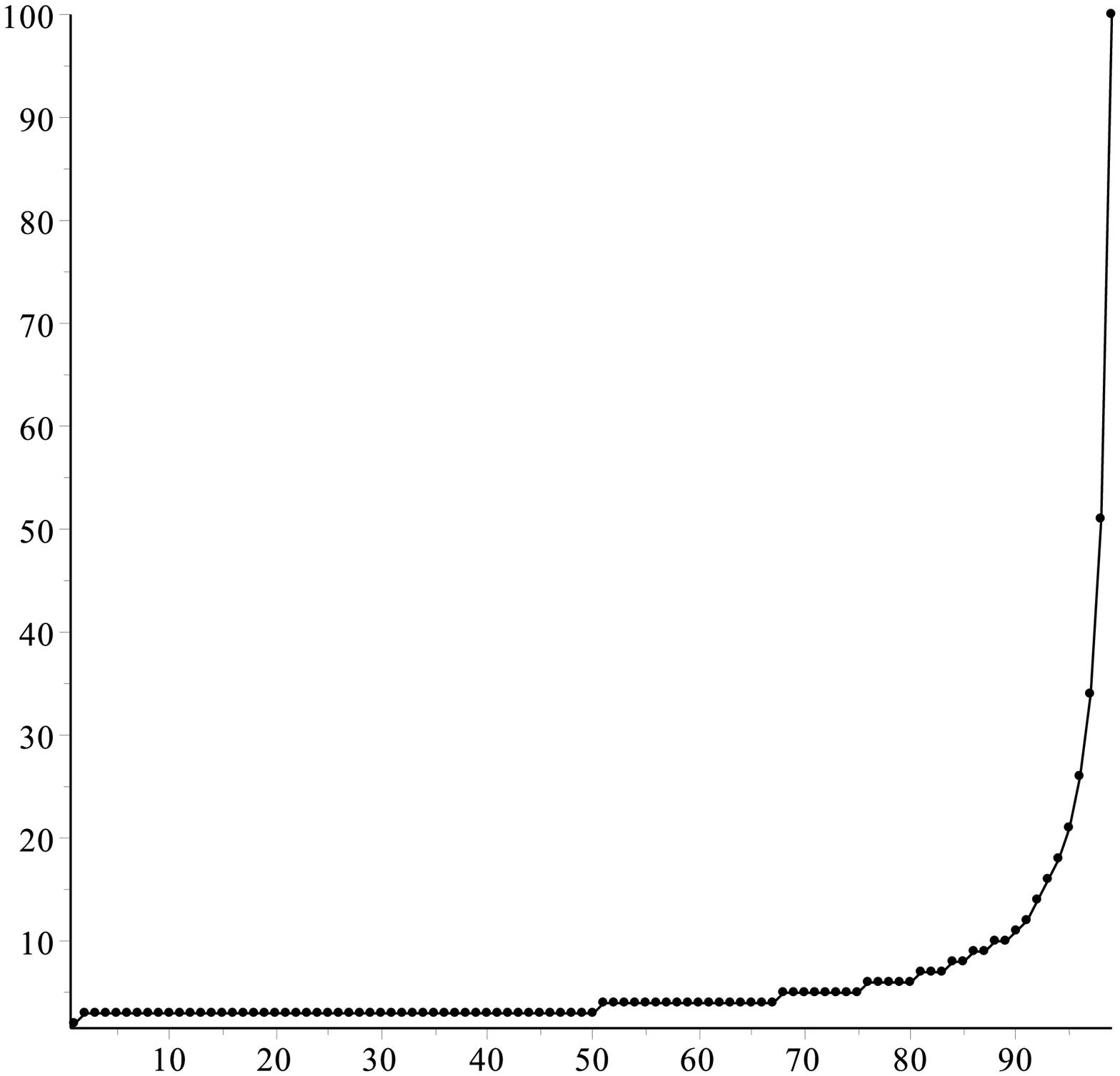}
\caption{$\tr^m$ as a function of $a$ for games with $m=25,50,100$ (left, middle and right figures respectively)}
\end{figure}

For each $m=\max S$, shrink the $\tr^m$ sequence to the set $x^m=x=\{\tr^m_a\}$ and enumerate the elements in increasing order; we interpret $x$ as a sequence $x^m=(x_a)$ with $x_1 = 1$ (by the theorem for full support) and $\max x^m = m-1$ (by the support size 2 result). We have, for all $a\ge 1$, $x_a<x_{a+1}$. 
But, what is the number of elements in $x$, for each $m$? The initial sizes of these sets are displayed in the last column of the table, as $\# x$. 

Study the first differences $\Delta_a^m=x^m_{a+1}-x^m_a$, $a\ge 1$.

Define, for all $m\ge 3$, and for all $1\le j\le \#x$, $M_j := \#\{a\mid x^m_j = \tr^m_a\}$. 

One can prove the following result by combining methods and results in Theorem~\ref{thm:fullsup} and Theorem~\ref{thm:2actions}.

\begin{theorem}\label{thm:duality}
For $a\in \{2,\ldots , \lceil m/2\rceil\}$, $\tr^{m}_a=2$, and moreover, $\Delta^m_{m-1}=\tr^{m}_{m} - \tr^{m}_{m-1} = \lfloor m/2\rfloor = \#\{a\mid \tr^{m}_a=2\}=M_2$. 
\end{theorem}

This result reflects an emerging `duality' between individual games and sequences of games, which appears to continue in the inner regions of the pictures. We make the following conjecture.

\begin{conjecture}[Duality]\label{con:duality}
Consider any truncated CS.
\begin{itemize}
\item For all $m\ge 2$, $\# x^m = \left\lfloor \sqrt{4m-7} \right\rfloor $ (corresponding to sequence OEIS: A000267).
\item The first differences, $\Delta^m$, equal in reverse order the number of multiplicities of the numbers in $\tr^m$.  That is, for all $a$, $M_{m+1-a}=\Delta^m_a$.
\end{itemize}
\end{conjecture}
Consider for example $\tr^{10}$. Then $\Delta = (1,1,2,4)$, and $M=(4,2,1,1)$. Careful inspection reveals that the pictures for $m=25,50,100$ satisfy this precise correspondence (and we checked many cases up to $m=200$), but we have up to date no means of explaining this proposed `duality'.

\section{Discussion}
In our work we study four \textit{classes} of CS games,  all with a finite support.  The convergence theorem, Theorem~\ref{Thm:convergence}, tells how to play for any given game with large heap size;  however when the heap size is small, we only have full understanding of optimal play in the classes of 2-actions and full support. As future work we suggest to study  optimal play when heaps are small, in other classes of CS, in particular the class of truncated games.\footnote{A \emph{class} of games should be described by a small finite number of game parameters (on the action set), and the optimal play solution should be described in terms of these parameters only.}

\begin{figure}[htbp]
\centering
    \includegraphics[width=6cm,height=6cm]{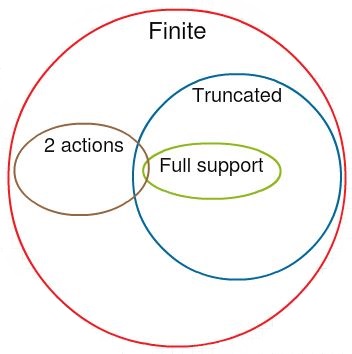}
    \caption{CS classes studied in this paper.} \label{fig:classification}
\end{figure}

For the case $|S|=2$, Theorem~\ref{thm:2actions} states the positions where it is optimal to sacrifice.  The following two observations are immediate from this result. 
\begin{observation}\label{obs:parity_advantage}
    Consider a game with exactly two actions. In optimal play, if a player makes a sacrifice, then she plays the last move. 
\end{observation}

\begin{observation}\label{obs:one player plays greedy} Consider a game with exactly two actions. In optimal play, at least one of the players plays only greedy actions.  
\end{observation}


For games with more than 2 possible actions the observations do not hold any more.
\begin{example}\label{ex:1,5,7}
Let $S=\{1,5,7\}$ played from position $x=18$.  The (unique) optimal play sequence is $5;7;5;1$, showing that sometimes it is beneficial to sacrifice without playing last.  Actually Positive sacrifices in order to play the last `big' action.
\end{example}

For games with $|S| \ge 4$ It is not true that only one player sacrifices in optimal play. Consider the following example
\begin{example} \label{ex:2,10,13,14}
    Let $S=\{2,10,13,14\}$ played from position $x=35$.
   The unique optimal play sequence is $10;13;10;2$, and the first two actions are both sacrifices. 
\end{example} 

This example triggers another question: is it true that when both players sacrifice, Negative makes a smaller sacrifice than Positive?
\begin{conjecture}
    In a game were optimal play includes sacrifices by both players, Negative's sacrifice is smaller than Positive's sacrifice. (In Example~\ref{ex:2,10,13,14} Positive sacrifices 4 while Negative sacrifices 1).
\end{conjecture}

\section{Multi pile CS} 
CS can be extended naturally to multiple piles.  In CS with multiple piles, on each turn the active player first chooses a pile, then plays as in the single pile game on that pile.  (In the CGT jargon, this is disjunctive sum play.) 

By looking at many games with two piles such as in Figures~\ref{fig:2piles57} and \ref{fig:2piles2,10,13,14} we observe convergence to the greedy action and periodicity in the outcome. By using similar arguments as in the folklore for classical subtracting games, one can show that in CS with two piles, the outcome is eventually periodic on any horizontal or vertical line. Here, we strengthen this result to a conjecture in the spirit of Theorem~\ref{Thm:convergence}.
\begin{conjecture}
    Consider CS on two piles. The outcome is eventually periodic on any horizontal or vertical line, with period at most $2\max{S}$. 
\end{conjecture}

In addition, we observe regularity of the outcomes along diagonal half-lines of the form $(x,k+x)$, called $k$-diagonals, for any constant $k\in \Z$. 
\begin{conjecture}
    Consider CS on two piles.  The outcome is eventually periodic along any $k$-diagonal.
\end{conjecture}

\begin{figure}[htbp]
\centering
   \includegraphics[width=13cm,height=13cm]{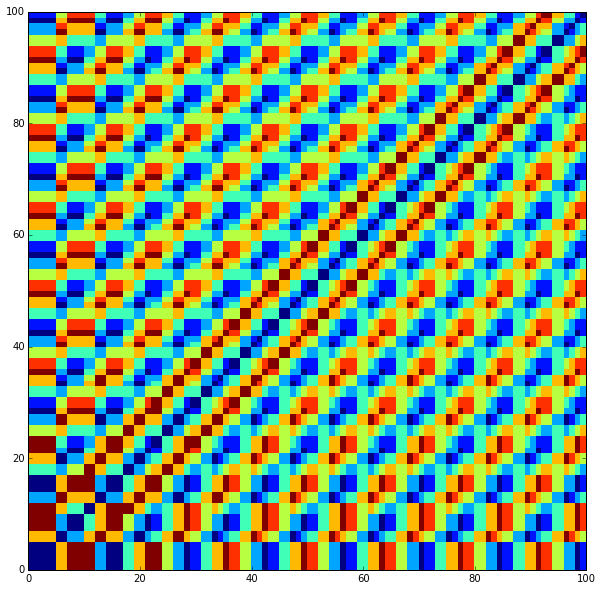}

    \caption{Outcomes for the game $S=\{5,7\}$ with two piles, starting from a position of the form $(x_1,x_2)$. The outcome is bounded between 0 and 7, low outcomes are painted in blue while high outcomes in red.  \label{fig:2piles57}}
\end{figure}

\begin{figure}[htbp]
\centering
    \includegraphics[width=13cm,height=13cm]{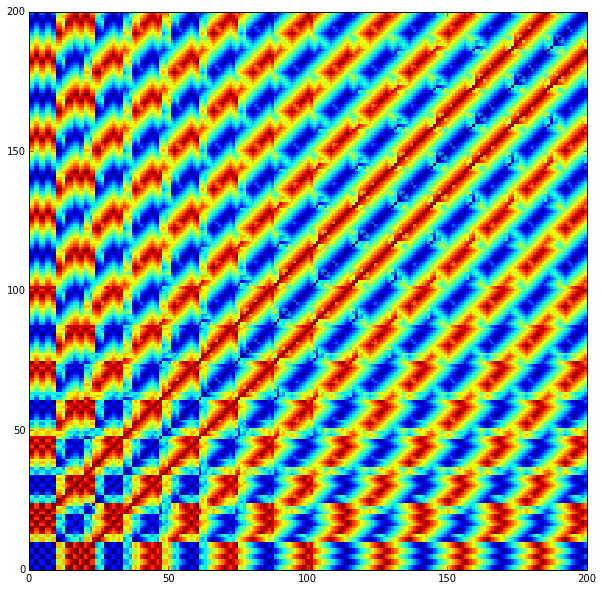}
    \caption{Outcomes for the game $S=\{2, 10, 13, 14\}$ with two piles, starting from a position of the form $(x_1,x_2)$. The outcome is bounded between 0 and 14, low outcomes are painted in blue while high outcomes in red.} \label{fig:2piles2,10,13,14}
\end{figure}


\end{document}